\documentclass[12pt]{article}
\usepackage[T2A]{fontenc}
\usepackage[centertags]{amsmath}
\usepackage{amsthm}
\usepackage{amsfonts}
\usepackage{amssymb}
\usepackage{mathrsfs}
\usepackage{listings,xcolor}

\def\b{\mathbf{b}}
\def\r{\mathbf{r}}
\def\q{\mathbf{q}}
\def\N{\mathbb{N}}

\newtheorem{lemma}{\hspace*{\parindent}Lemma}
\newtheorem{theorem}{\hspace*{\parindent}Theorem}

\newtheorem{corollary}{\hspace*{\parindent}Corollary}
\title{Hypergeometric ${}_4F_3(1)$ with integral parameter differences}
\author{D.B.\:Karp$^{\rm a}$\footnote{Corresponding author. E-mail: D. Karp -- \emph{dimkrp@gmail.com},
E.\:Prilepkina --  \emph{pril-elena@yandex.ru}}~~and
E.G.\:Prilepkina$^{\rm a,b}$
\\[10pt]
\small{\textit{$\phantom{1}^a$Holon Institute of Technology, Holon, Israel}}
\\
\small{\textit{$\phantom{1}^b$Far Eastern Federal University, Ajax Bay~10, 690922 Vladivostok, Russia}}
\\
\small{\textit{$\phantom{1}^c$Institute of Applied Mathematics,
FEBRAS, 7 Radio Street, Vladivostok,  690041, Russia}}}
\date{}
\date{}
\begin{document}
\maketitle

\begin{abstract} 
In this  paper we continue investigation of the hypergeometric function ${}_4F_3(1)$ as the function of its seven parameters.  We deduce several reduction formulas for this function under additional conditions that one of the top parameters exceeds one of the bottom parameters by a positive integer or, conversely, one of the bottom parameters exceeds one of the top parameters by a positive integer or both.  We show that all such cases reduce to the case of the unit parameter difference. The latter case, in turn, can  be expressed in terms of certain linear combination of two series involving the logarithmic derivative of the gamma function. 
\end{abstract}

\bigskip

Keywords: \emph{generalized hypergeometric function, hypergeometric identity, integral parameter difference, summation formula}

\bigskip

MSC2010: 33C20

\bigskip

\section{Introduction}

The generalized hypergeometric functions ${}_pF_q(a_1,\ldots,a_p; b_1,\ldots, b_q;z)$ \cite[Chapter~3]{AAR},\cite{Bailey} occur in a wide variety of problems in theoretical physics, applied mathematics, statistics and engineering sciences, let along the pure mathematics itself.  In particular, the functions  ${}_3F_2$ and ${}_4F_3$ evaluated at the unit argument are related to the Clebsch-Gordan and Racah coefficients, respectively, see \cite[sections~8.2.5 and 9.2.3]{VMK} and \cite{KrRao2004,RaoDobermann}. The celebrated summation formula due to Minton \cite{Minton} for the generalized hypergeometric series with integral parameter differences and it generalization by Karlsson \cite{Karlsson} motivated a stream of works dedicated to this type of hypergeometric series. Extensions in many directions were found, including a $q$-analogue and a generalization of Minton's and Karlsson's formulas due to Gasper \cite{Gasper}; bilateral hypergeometric and $q$-hypergeometric series due to Chu \cite{Chu94,Chu95}; these results were re-derived by simpler means and further generalized by Schlosser \cite{Schlosser2002}, who also found multidimensional extensions to hypergeometric functions associated with root systems \cite{Schlosser2001}.

In a series of papers \cite{KPITSF2018,KPREsults2019,KPSpringer2020,KPG} we studied transformation and summation formulas for the generalized hypergeometric functions with integral parameter differences.  In particular, in \cite{KPG} we gave a complete description of  the group of two-term relations for the function 
\begin{equation}\label{eq:4F3unitshift}
{}_4F_3\!\left(\!\!\begin{array}{c}
    a,b,c,f+1\\
    d,e,f
\end{array}\!\!\right):={}_4F_3\!\left(\!\!\left.\begin{array}{c}
    a,b,c,f+1\\
    d,e,f
\end{array}\right| 1\right)
\end{equation}
(we will routinely omit argument $1$ from the notation of the hypergeometric function).  In this paper we will investigate the function ${}_4F_3(1)$ containing the parameter pair  $\Bigg[\!\!\begin{array}{c}f+m\\[-2pt]f\end{array}\!\!\Bigg]$ (known as the positive integral parameter difference) and/or $\Bigg[\!\begin{array}{c}b\\[-2pt]b+n\end{array}\!\Bigg]$ (known as the negative integral parameter difference) for arbitrary positive integers $m$, $n$. Our main result states that each function containing both pairs can be reduced to a single  ${}_3F_2(1)$ containing the parameter pair $\Bigg[\!\begin{array}{c}b\\[-2pt]b+1\end{array}\!\Bigg]$ modulo rational and gamma functions.   Via Thomae's transformation one can use  ${}_3F_2(1)$ with one top parameter equal to unity instead.  
This seems to be generally accepted (although we have not seen a formal proof) that this type of ${}_3F_2(1)$ is not summable in terms of gamma functions. This and related problems have been studied recently by various authors, see, for instance, \cite{Chen,EbisuIwasaki,SS}.   Combined with known summation formulas our results lead to new summability cases for ${}_4F_3(1)$ in terms of gamma functions.  We further demonstrate that that certain type of series containing the logarithmic derivative $\psi(x)$ of the gamma function reduces to the same type of ${}_3F_2(1)$ and, hence,  general ${}_4F_3(1)$ with one positive and one negative parameter difference reduces to this type of $\psi$-series. This paper is organized as follows: in the following Section~2 we treat the case of one  positive parameter difference; Section~3 is devoted to negative parameter difference as well as combination of one positive and one negative parameter difference; concluding Section~4 reveals the relation with $\psi$-series.

\section{Positive parameter difference}
Before formulating our first proposition, let us fix the following convention that will be valid throughout the paper: \emph{in all our statements involving infinite series we assume the parameters are chosen so that all such series converge and, furthermore, the rational and gamma functions appearing in the formulas do not become infinite}.  
Such parameter ranges are not empty in all our statements as can be immediately seen from the well-known convergence properties of the hypergeometric series.  

Our first lemma is the key ingredient in reducing the ${}_4F_3(1)$ with an arbitrary positive parameter difference to the case of the unit shift.
\begin{lemma}\label{lem1}
The following reduction formula holds\emph{:} 
\begin{equation}\label{eq:plus2}
{}_4F_3\!\left(\!\!\begin{array}{c}
    a,b,c,f+2\\
    d,e,f
\end{array}\!\right)=\left(1+\frac{abc}{(s-1)f(f+1)}\right){}_4F_3\!\left(\!\!\begin{array}{c}
    a,b,c,\mu+1\\
    d,e,\mu
\end{array}\!\right),
\end{equation} 
where $s=d+e-a-b-c-1$  and
\begin{equation}\label{eq:mu2}
\mu=\frac{abc+(s-1)f(f+1)}{ab+ac+bc-de+d+e+(s-1)(2f+1)-1}.
\end{equation}
\end{lemma}
\begin{proof} Directly from the definition of the hypergeometric series we get the expansion
\begin{multline}{\label{eq:ext1}}
{}_4F_3\!\left(\!\!\begin{array}{c}
    a,b,c,f+2\\
    d,e,f
\end{array}\!\right)={}_5F_4\!\left(\!\!\begin{array}{c}
    a,b,c,f+1,f+2\\
    d,e,f,f+1
\end{array}\!\!\right)={}_4F_3\!\left(\!\!\begin{array}{c}
    a,b,c,f+1\\
    d,e,f
\end{array}\!\!\right)
\\
+\dfrac{abc}{efd}{}_4F_3\!\left(\!\!\begin{array}{c}
    a+1,b+1,c+1,f+2\\
    d+1,e+1,f+1
\end{array}\!\!\right).
\end{multline}
Application of the transformations \cite[page~10 and formula (A2)]{KPG} yields the following expression for the rightmost term: 
\begin{equation}\label{eq:shift1}
  {}_4F_3\!\left(\!\!\begin{array}{c}
    a+1,b+1,c+1,f+2\\
    e+1,d+1,f+1
\end{array}\!\!\right)=\frac{d e}{(s-1)(f+1)}{}_4F_3\!\left(\!\!\begin{array}{c}
    a,b,c,\eta+1\\
    d,e,\eta
\end{array}\!\!\right),  
\end{equation}
where
$$
\eta=\frac{abc}{ab+ac+bc-de+d+e+(s-1)f-1}.
$$
Substituting this back into \eqref{eq:ext1}, we will have 
\begin{equation*}
{}_4F_3\!\left(\!\!\begin{array}{c}
    a,b,c,f+2\\
    d,e,f
\end{array}\!\!\right)={}_4F_3\!\left(\!\!\begin{array}{c}
    a,b,c,f+1\\
    d,e,f
\end{array}\!\!\right)+\frac{abc}{(f+1)f(s-1)}{}_4F_3\!\left(\!\!\begin{array}{c}
    a,b,c,\eta+1\\
    d,e,\eta
\end{array}\!\!\right).
\end{equation*} 
The claim of the lemma results in expanding both terms and writing the sum as the single series in view of $(f+1)_n/(f)_n=1+n/f$.
\end{proof}

Denoting the coefficient in \eqref{eq:plus2} by
\begin{equation}\label{eq:W2}
W_2=W_2(a,b,c,d,e,f)=1+\frac{abc}{(s-1)f(f+1)}
\end{equation}
and writing $\mu_2=\mu(a,b,c,d,e,f)$ for $\mu$ from  \eqref{eq:mu2} we can rewrite \eqref{eq:plus2} in the form
\begin{equation}\label{eq:m2}
{}_4F_3\!\left(\!\!\begin{array}{c}
    a,b,c,f+2\\
    d,e,f
\end{array}\!\!\right) =W_2(a,b,c,d,e,f){}_4F_3\!\left(\!\!\begin{array}{c}
    a,b,c,\mu_2+1\\
    d,e,\mu_2
\end{array}\!\!\right).
\end{equation}
Iterating this formula we get the following identity for an arbitrary integer shift $m>0$.
\begin{theorem}\label{th:1} 
For each $m\in\N$ there exist rational functions $W_m=W_m(a,b,c,d,e,f)$ and $\mu_m=\mu_m(a,b,c,d,e,f)$ such that 
 $$
{}_4F_3\!\left(\!\!\begin{array}{c}
    a,b,c,f+m\\
    d,e,f
\end{array}\!\!\right)=W_m(a,b,c,d,e,f)\,{}_4F_3\!\left(\!\!\begin{array}{c}
    a,b,c,\mu_m+1\\
    d,e,\mu_m
\end{array}\!\!\right).
$$  
The functions $W_m=W_m(a,b,c,d,e,f)$ and $\mu_m=\mu_m(a,b,c,d,e,f)$ are computed recursively by 
$$
W_{m+1}=W_m+\frac{abc\nu_m}{(s-1)f\sigma_m},
$$
$$
\mu_{m+1}=W_{m+1}\left(\frac{W_m}{\mu_m}+\frac{abc\nu_m}{(s-1)f\sigma_m\eta_m}\right)^{-1},
$$
where
\begin{align*}
&\nu_m=W_m(a+1,b+1,c+1,d+1,e+1,f+1),
\\[3pt]
&\sigma_m=\mu_m(a+1,b+1,c+1,d+1,e+1,f+1),
\\[3pt]
&\eta_m=\frac{abc}{ab+ac+bc-de+d+e+(s-1)(\sigma_m-1)-1},
\end{align*}
and the initial values  $W_2$,  $\mu_2$ are given in \eqref{eq:mu2}, \eqref{eq:W2}.
\end{theorem}

\begin{proof} 
We will use induction in $m$. The claim is true for $m=2$ according to \eqref{eq:m2}. Suppose that 
$$
{}_4F_3\!\left(\!\!\begin{array}{c}
    a,b,c,f+m\\
    d,e,f
\end{array}\!\!\right) =W_m(a,b,c,d,e,f){}_4F_3\!\left(\!\!\begin{array}{c}
    a,b,c,\mu_m+1\\
    d,e,\mu_m
\end{array}\!\!\right).
$$  
According to the definition of the hypergeometric series and using the property $(f+m+1)_n/(f+m)_n=1+n/(f+m)$ of the rising factorial, we will get 
\begin{multline}
{}_4F_3\!\left(\!\!\begin{array}{c}
    a,b,c,f+m+1\\
    d,e,f
\end{array}\right)={}_5F_4\!\left(\!\!\begin{array}{c}
    a,b,c,f+m,f+m+1\\
    d,e,f, f+m
\end{array}\!\!\right)
\\ 
={}_4F_3\!\left(\!\!\begin{array}{c}
    a,b,c,f+m\\
    d,e,f
\end{array}\!\!\right)
+\frac{abc}{def}{}_4F_3\!\left(\!\!\begin{array}{c}
    a+1,b+1,c+1,f+m+1\\
    d+1,e+1,f+1
\end{array}\!\!\right).
\end{multline}
Applying the induction hypothesis to the first term we will have
\begin{multline}\label{eq:decom}
{}_4F_3\!\left(\!\!\begin{array}{c}
    a,b,c,f+m+1\\
    d,e,f
\end{array}\!\!\right) =W_m\cdot{}_4F_3\!\left(\!\!\begin{array}{c}
    a,b,c,\mu_m+1\\
    d,e,\mu_m
\end{array}\!\!\right)
\\
+\frac{abc\nu_m}{def}{}_4F_3\!\left(\!\!\begin{array}{c}
    a+1,b+1,c+1,\sigma_m+1\\
    d+1,e+1,\sigma_m
\end{array}\!\!\right).
\end{multline}
On the other hand, setting $f=\sigma_m-1$ in \eqref{eq:shift1} yields 
$$
{}_4F_3\!\left(\!\!\begin{array}{c}
    a+1,b+1,c+1,\sigma_m+1\\
    e+1,d+1,\sigma_m
\end{array}\!\right)
=\frac{d e}{(s-1)\sigma_m}{}_4F_3\!\left(\!\!\begin{array}{c}
    a,b,c,\eta_m+1\\
    d,e,\eta_m
\end{array}\!\!\right),
$$ 
with $\eta_m$ defined in the formulation of the theorem.
Then, in view of \eqref{eq:decom}, we obtain the chain of equalities:
\begin{multline*}
{}_4F_3\!\left(\!\!\begin{array}{c}
    a,b,c,f+m+1\\
    d,e,f
\end{array}\!\!\right)
=W_m\cdot{}_4F_3\!\left(\!\begin{array}{c}
    a,b,c,\mu_m+1\\
    d,e,\mu_m
\end{array}\right)+\frac{abc\nu_m}{(s-1) f\sigma_m}{}_4F_3\!\left(\!\!\begin{array}{c}
    a,b,c,\eta_m+1\\
    d,e,\eta_m
\end{array}\!\!\right)
\\
=\sum\limits_{n=0}^\infty \frac{(a)_n(b)_n (c)_n}{(e)_n (d)_n n!}\left(\left(W_m+\frac{abc\nu_m}{(s-1)f\sigma_m}\right)+n\left(\frac{W_m}{\mu_m}+\frac{abc\nu_m}{(s-1)f\sigma_m\eta_m}\right)\right)
\\
=W_{m+1}\sum\limits_{n=0}^\infty \frac{(a)_n(b)_n (c)_n}{(e)_n (d)_n n!} (1+n/\mu_{m+1})
=W_{m+1}\cdot{}_4F_3\!\left(\!\!\begin{array}{c}
    a,b,c,\mu_{m+1}+1\\
    d,e,\mu_{m+1}
\end{array}\!\!\right).
\end{multline*} 
\end{proof} 

In our recent paper \cite[Proposition~5]{KPG} we have proposed an algorithm for obtaining summation formulas for the ${}_4F_3(1)$ function of the form \eqref{eq:4F3unitshift} with non-linearly constrained parameter $f$.   Now Theorem~\ref{th:1} can be used to extend our method to ${}_4F_3(1)$ containing the parameter pair  $\Bigg[\!\!\begin{array}{c}f+m\\[-2pt]f\end{array}\!\!\Bigg]$.  Let's demonstrate this approach with the following example. In \cite[(A.4)]{KPG}  we established the transformation 
\begin{equation}\label{eq:killbottom-1}
{}_{4}F_3\!\!\left(\begin{matrix}a,b,c,\mu+1\\d,e,\mu\end{matrix}\right)
=\frac{[((d-b-1)(d-c-1)-a(d-b-c-1))\mu-abc](d-1)}{(d-a-1)(d-b-1)(d-c-1)\mu}
{}_{4}F_3\!\!\left(\begin{matrix}a,b,c,\eta+1\\d-1,e,\eta\end{matrix}\right),
\end{equation}
where
$$
\eta=\frac{abc+[(1-d)(d-a-b-c-1)-ab-ac-bc]\mu}{(d+e-a-b-c-2)(\mu-d+1)}.
$$
Denote by $s_3=abc,$ $s_2=ab+ac+bc.$  Direct application of the formulas \eqref{eq:plus2},\eqref{eq:killbottom-1}, \cite[(41a)]{KPG}   leads to the following assertion. 
\begin{corollary}
Assume that $d+e=a+b+c+3$ and $f$ satisfies the equation
\begin{equation}
\frac{s_3}{s_2-(2-d)(1-e)}=\frac{s_3(s_2-de+d+e+2f)+(s_3+f^2+f)(2-e-2d+de-s_2)}{(s_3+f^2+f)+(1-d)(s_2-de+d+e+2f)}.
\end{equation}
Then 
 \begin{equation}\label{eq:sumplus2}
{}_4F_3\left(\begin{array}{c}
    a,b,c,f+2\\
    d,e,f
\end{array}\right)=\frac{[(d-b-1)(d-c-1)\mu-a(d-b-c-1)\mu-s_3]\Gamma(d)\Gamma(e)((f)_2+s_3)}{s_3(f)_2(d-a-1)(d-b-1)(d-c-1)\mu\Gamma(a)\Gamma(b)\Gamma(c)},
\end{equation} 
where \begin{equation}\mu=\frac{s_3+f^2+f}{s_2-de+d+e+2f}.  
\end{equation} 
\end{corollary}

\section{Negative parameter difference}

Using the the partial fraction decomposition of the ratio $(b)_n/(b+m)_n$ for $m\in\N$, $m\ge2$, given in the proof of \cite[Theorem 2.2]{{KPITSF2018}},  we immediately obtain the expansion 
\begin{equation}\label{eq:dec2}
{}_4F_3\!\left(\!\!\begin{array}{c}
    a,b,c,g\\
    b+m,d,f
\end{array}\!\!\right)=\sum\limits_{q=0}^{m-1}A_q\cdot {}_4F_3\!\left(\!\!\begin{array}{c}
    a,b+q,c,g\\
    b+q+1,d,f
\end{array}\!\!\right),
\end{equation}
where 
$$
A_q=\frac{(b)_m}{(b+q)_m}\prod_{\substack{l=1\\l\not= q+1}}^m(l-q-1)^{-1}.
$$ 
Each term on the right hand side of the above expansion has unit negative parameter shift. It seems natural to ask whether it is possible to ''wrap up'' the expression on the right hand side of \eqref{eq:dec2} into a single  hypergeometric series $_4F_3(1)$. The goal of this section is to answer this question in the affirmative under the additional assumption $g=f+1$.  In view of the results of the previous section, the more general case $g=f+m$, $m\in\N$,  reduces to the one treated here. 
The main result is  the following claim.
\begin{theorem}\label{th:RQtheorem}
There exist rational functions $R_m=R_m(a,b,c,e,f)$ and  $Q_m=Q_m(a,b,c,e,f)$  such that for any integer $m\ge2$ we have
\begin{equation}\label{eq:RQtheorem}
{}_{4}F_3\!\!\left(\begin{matrix}a,b,c,f+1\\b+m,e,f\end{matrix}\right)\!=\!R_m(a,b,c,e,f){}_{4}F_3\!\!\left(\begin{matrix}a,b,c,f+1\\b+1,e,f\end{matrix}\right)+Q_m(a,b,c,e,f)\frac{\Gamma(e)\Gamma(e-a-c-1)}{\Gamma(e-a)\Gamma(e-c)}.
\end{equation}
The functions $R_m$, $Q_m$ can be computed recursively according to the following formulas\emph{:}
$$
R_m=-\frac{B(m-2)}{A(m-2)}R_{m-1}-\frac{C(m-2)}{A(m-2)}R_{m-2},
$$
$$
Q_m=-\frac{B(m-2)}{A(m-2)}Q_{m-1}-\frac{C(m-2)}{A(m-2)}Q_{m-2}
$$
with initial values $R_{0}=0$, $R_{1}=1$, $Q_{0}=e-a-c-1+ac/f$, $Q_{1}=0$ and $A(k)$, $B(k)$, $C(k)$ given in \eqref{eq:3termcoeff}.
\end{theorem}
It follows immediately from \cite[(38)]{KPG} and the Gauss summation formula that
$$
{}_{4}F_3\!\!\left(\begin{matrix}a,b,c,f+1\\b+1,e,f\end{matrix}\right)=\frac{a(b-f)}{f(b-a)}{}_{3}F_2\!\!\left(\begin{matrix}a+1,b,c\\b+1,e\end{matrix}\right)+\frac{b(f-a)\Gamma(e)\Gamma(e-a-c)}{f(b-a)\Gamma(e-a)\Gamma(e-c)}.
$$
Substituting this expression on the right hand side of \eqref{eq:RQtheorem} we can further reduce it to ${}_{3}F_2(1)$ with negative unit shift. Theorem~\ref{th:RQtheorem} will follows immediately from the recurrence relation given next and its $k=0$ particular case given in Corollary~\ref{cr:negative}.  
\begin{theorem}\label{th:BottomRec}
The following three-term recurrence relation holds:
\begin{equation}\label{eq:3term}
A(k){}_{4}F_3\!\!\left(\begin{matrix}a,b,c,f+1\\b+k+2,e,f\end{matrix}\right)+
B(k){}_{4}F_3\!\!\left(\begin{matrix}a,b,c,f+1\\b+k+1,e,f \end{matrix}\right)
+C(k){}_{4}F_3\!\!\left(\begin{matrix}a,b,c,f+1\\b+k,e,f \end{matrix}\right)\\
=0,
\end{equation}
where the coefficients are given by 
\begin{equation}\label{eq:3termcoeff}
\begin{split}
A(k)&=-\frac{\alpha_k(k+1)(b-a+k+1)(b-c+k+1)}{(b+k+1)},    
\\[7pt]
B(k)&=\beta_k(b-a+k)(b-f+k+1)(e-a-c+k)
\\
+&\left((k+1)(a-f)(b-c+k+1)-a(a-e+1)(b-f+k+1)\right)
\\
&\times\left(\beta_k+k(b-c+k)\right),
\\[7pt]
C(k)&=-\alpha_{k+1}(b+k)(e-a-c-1+k)
\end{split}
\end{equation}
with 
\begin{align*}
\alpha_k&=(b-f)[a(c-f)+f(e-c-1)]+kf(e-f-1),    
\\[5pt]
\beta_k&=a(c-k)+f(e-a-c-1+k).
\end{align*}
\end{theorem}
\textbf{Proof}.  We will derive relation \eqref{eq:3term} from the following three-term relation for ${}_3F_2(1)$:
\begin{multline}\label{eq:3term3F2}
{}_{3}F_2\!\!\left(\begin{matrix}a+1,b,c\\b+k+1,e\end{matrix}\right)
=\underbrace{\frac{a-e+1}{a+c-e-k}}_{R_2}{}_{3}F_2\!\!\left(\begin{matrix}a,b,c\\b+k+1,e\end{matrix}\right)
\\
+\underbrace{\frac{(k+1)(b-c+k+1)}{(b+k+1)(-a-c+e+k)}}_{R_1}
{}_{3}F_2\!\!\left(\begin{matrix}a+1,b,c\\b+k+2,e\end{matrix}\right).
\end{multline}
This formula can be verified by combining \cite[(31)]{KPG} and \cite[(32)]{KPG} with the formula 
\begin{equation}\label{eq:inverse}
	{}_{4}F_3\!\!\left(\begin{matrix}a,b,c, f+1\\d,e,f\end{matrix}\right)
	=\frac{f(a+b+c-d-e+1)-bc}{f(a+b+c-d-e+1)}{}_{4}F_3\!\!\left(\begin{matrix}a-1,b,c,\mu+1\\d,e,\mu\end{matrix}\right),
\end{equation}
given in \cite[(A2)]{KPG}, and replacing $d\to b+k+1$ and $a\to a+1$ in the resulting identity. The parameter $\mu$ in \eqref{eq:inverse} is given by 
$$
\mu=\frac{(a-1)[(a+b+c-d-e+1)f-bc]}{(a+b+c-d-e+1)f-bc-(a-1)(d+e-a-1)+de-d-e+1}.
$$
Alternatively, one can apply the algorithm elaborated in \cite{EbisuIwasaki} which permits finding coefficients in any three-term relation connecting contiguous  ${}_3F_2(1)$ to derive \eqref{eq:3term3F2}.

Our next goal is to expand each ${}_4F_{3}$  term in \eqref{eq:3term} into the sum of ${}_3F_{2}$ terms of the form appearing in \eqref{eq:3term3F2}.
First, writing $a$ for $\alpha-1$ and $b+k+1$ for $d$ in \cite[(37)]{KPG} we get
\begin{equation}\label{eq:b+k+1}
{}_{4}F_3\!\!\left(\begin{matrix}a,b,c,\xi+1\\b+k+1,e,\xi\end{matrix}\right)
=\frac{a}{\xi}{}_{3}F_2\!\!\left(\begin{matrix}a+1,b,c\\b+k+1,e\end{matrix}\right)
+\frac{\xi-a}{\xi}{}_{3}F_2\!\!\left(\begin{matrix}a,b,c\\b+k+1,e\end{matrix}\right),
\end{equation}
which yields the middle term in \eqref{eq:3term} on writing $\xi=f$.  Denote the coefficients in the resulting formula by $\gamma=(f-a)/a$, $\delta=f/a$. 

Next, substituting $d=b+k$ into transformation \cite[(31)]{KPG} we get
\begin{equation}\label{eq:bottom+1}
{}_{4}F_3\!\!\left(\begin{matrix}a,b,c,f+1\\b+k,e,f\end{matrix}\right)
=\frac{fs(b+k)+abc}{fs(b+k)}
{}_{4}F_3\!\!\left(\begin{matrix}a,b,c,\nu+1\\b+k+1,e,\nu\end{matrix}\right),
\end{equation}
where $s=k+e-a-c-1$ and
$$
\nu=\frac{fs(b+k)+abc}{fs+k(b+k-c)+a(c-k)}.
$$
Hence, taking $\xi=\nu$ in \eqref{eq:b+k+1} and substituting it back into \eqref{eq:bottom+1}, we arrive at 
\begin{equation}\label{eq:b+k}
{}_{4}F_3\!\!\left(\begin{matrix}a,b,c,f+1\\b+k,e,f\end{matrix}\right)
=\lambda(k){}_{3}F_2\!\!\left(\begin{matrix}a+1,b,c\\b+k+1,e\end{matrix}\right)
+\eta(k){}_{3}F_2\!\!\left(\begin{matrix}a,b,c\\b+k+1,e\end{matrix}\right),
\end{equation}
where
$$
\lambda(k)=\frac{a\left(fs+k(b+k-c)+a(c-k)\right)}{fs(b+k)},~~~~\eta(k)=\frac{(a-b-k)\left(a(k-c)-fs\right)}{fs(b+k)},
$$
which is the required expression for the rightmost term in \eqref{eq:3term}.
Finally, substituting $a$ for $\alpha$ and $b+k+1$ for $d$ in \cite[(38)]{KPG} brings it to the form:
\begin{equation}\label{eq:b+k+2}
{}_{4}F_3\!\!\left(\begin{matrix}a,b,c,f+1\\b+k+2,e,f\end{matrix}\right)=\beta(k){}_{3}F_2\!\!\left(\begin{matrix}a,b,c\\b+k+1,e\end{matrix}\right)
+\alpha(k) {}_{3}F_2\!\!\left(\begin{matrix}a+1,b,c\\b+k+2,e\end{matrix}\right),
\end{equation}
where
$$
\alpha(k)=\frac{a(b-f+k+1)}{f (b-a+k+1)}, ~~~~\beta(k)=\frac{(f-a)(b+k+1)}{f(b-a+k+1)},
$$
which is the required expression for the leftmost term in \eqref{eq:3term}.

Suppose now that 
\begin{equation*}
	A\cdot{}_{4}F_3\!\!\left(\begin{matrix}a,b,c,f+1\\b+k+2,e,f\end{matrix}\right)+ B\cdot{}_{4}F_3\!\!\left(\begin{matrix}a,b,c,f+1\\b+k+1,e,f \end{matrix}\right)
	+C\cdot{}_{4}F_3\!\!\left(\begin{matrix}a,b,c,f+1\\b+k,e,f\end{matrix}\right)=0,
\end{equation*}
with unknown coefficients $A$, $B$, $C$. Substituting the corresponding decompositions \eqref{eq:b+k+1}, \eqref{eq:b+k}, \eqref{eq:b+k+2} we will have:
$$
A\alpha{}_{3}F_2\!\!\left(\begin{matrix}a+1,b,c\\b+k+2,e\end{matrix}\right)
+(A\beta+B\gamma+C\eta){}_{3}F_2\!\!\left(\begin{matrix}a,b,c\\b+k+1,e\end{matrix}\right)
+(B\delta+C\lambda){}_{3}F_2\!\!\left(\begin{matrix}a+1,b,c\\b+k+1,e\end{matrix}\right)=0.
$$
Comparing this formula with \eqref{eq:3term3F2} and equating coefficients leads to the following system of linear equations for  $A$, $B$, $C$:
$$
A\alpha=R_1,~~~A\beta+B\gamma+C\eta=R_2,~~~ B\delta+C\lambda=-1
$$
with $R_1$, $R_2$ defined in \eqref{eq:3term3F2}. Solving this system yields \eqref{eq:3termcoeff}. $\hfill\square$

Substituting $k=0$ in \eqref{eq:3term} and using the summation formula (obtained by decomposing ${}_{3}F_2(1)$ into the sum of two ${}_2F_{1}$ and applying the Gauss theorem) 
$$
{}_{3}F_2\!\!\left(\begin{matrix}a,c,f+1\\e,f\end{matrix}\right)=\frac{\Gamma(e)\Gamma(e-a-c-1)}{f\Gamma(e-a)\Gamma(e-c)}((e-a-c-1)f+ac),
$$
we arrive at:
\begin{corollary}\label{cr:negative}
The following identity holds true\emph{:}
\begin{multline*}
	{}_{4}F_3\!\!\left(\begin{matrix}a,b,c,f+1\\b+2,e,f\end{matrix}\right)=-\frac{B(0)}{A(0)}{}_{4}F_3\!\!\left(\begin{matrix}a, b, c, f+1\\b+1,e,f\end{matrix}\right)
	\\
	-\frac{C(0)\Gamma(e)\Gamma(e-a-c-1)}{A(0)f\Gamma(e-a)\Gamma(e-c)}((e-a-c-1)f+ac)
\end{multline*}
with coefficients defined in \eqref{eq:3termcoeff}.
\end{corollary} 
Let us conclude this section with an explicit example of combined application of Theorems~\ref{th:1},\ref{th:RQtheorem}, the remark above Theorem~\ref{th:BottomRec} and Thomae's transformation. After various rearrangements this leads to the following identity ($s=e-a-c$):
\begin{multline*}
	\frac{\Gamma(e-c)\Gamma(e-a)(f)_2}{\Gamma(e)\Gamma(e-a-c)(b)_{2}}{}_{4}F_3\!\left(\begin{matrix}a,b,c,f+2\\b+2,e,f\end{matrix}\right)
	=\frac{(1+f-e)_2}{(1+b-e)_2}+\left(1-\frac{(1+f-e)_2}{(1+b-e)_2}\right)
	\\
	\times\!\!\Bigg[\frac{b^2(1-s)\!+\!b(f+e-1)(s+1)\!-\!3b(e-1)\!-\!s(f+1)(e-2)}
	{b(b+f-2e+3)}\!-\!\frac{s(b-e+2)}{b(e\!-\!a)(e\!-\!c)(b+f-2e+3)}
	\\
\times\!\left(b^2(1-s)+bf(s+1)-b(e-2)-abc+(a-1)(c-1)(f+1)\right){}_{3}F_2\!\!\left(\begin{matrix}1,s+1,e-b-1\\e-c+1, e-a+1\end{matrix}\right)\!\!\Bigg].
\end{multline*}

\section{Relation to $\psi$-series}
Define
\begin{equation}\label{eq:2F1hat}
 {}_{2}\hat{F}_1\!\!\left(\left.\begin{matrix}a,c\\e\end{matrix}\right\vert b\right):=\sum_{k=1}^{\infty}\frac{(a)_k(c)_{k}}{(e)_kk!}(\psi(b+k)-\psi(b)).
\end{equation}
Note that writing ${}_{3}F_2^{(1)}$ for the derivative of ${}_{3}F_{2}$ with respect to the first upper parameter we will have
$$
{}_{2}\hat{F}_1\!\!\left(\left.\begin{matrix}a,c\\e\end{matrix}\right\vert b\right)={}_{3}F_2^{(1)}\!\!\left(\begin{matrix}b,a,c\\b,e\end{matrix}\right)
~~\text{and}~~
{}_{2}\hat{F}_1\!\!\left(\left.\begin{matrix}a,c\\e\end{matrix}\right\vert a\right)={}_{2}F_1^{(1)}\!\!\left(\begin{matrix}a,c\\e\end{matrix}\right).
$$
It is further well-known and is easy to see that parametric derivatives can be expressed in terms of the Kamp\'{e} de Feri\'{e}t hypergeometric function of two variables. In particular, according to \cite[(3.1)]{Cvijovic}, we have:
\begin{equation*}
{}_{2}\hat{F}_1\!\!\left(\left.\begin{matrix}a,c\\e\end{matrix}\right\vert b\right)=\frac{ac}{be}F^{2:2;1}_{2:1;0}\!\left(\!\left.\begin{matrix}a+1,c+1:1,b;1\\e+1, 2:b+1;-\end{matrix}\right\vert 1,1\!\right).
\end{equation*}

\begin{theorem}\label{th:4F3to2F1hat}
The following identity holds\emph{:}
\begin{multline}\label{eq:4F3to2F1hat}
{}_{4}F_3\!\!\left(\begin{matrix}a,b,c,f+1\\b+1,e,f\end{matrix}\right)
\!=\!\frac{abc(b-f)}{f(b-a)(b-c)}
\!\left\{\!{}_{2}\hat{F}_1\!\!\left(\left.\begin{matrix}a,c\\e\end{matrix}\right\vert b\!\right)
\!+\!\frac{1+a+c-e}{e}{}_{2}\hat{F}_1\!\!\left(\left.\begin{matrix}a+1,c+1\\e+1\end{matrix}\right\vert b+1\!\right)\!\right\}
\\
+\frac{f(b-c)(a^2+(a-b)(c-e+1))+ab^2(c-f)}{f(b-a)(b-c)}\frac{\Gamma(e)\Gamma(e-a-c-1)}{\Gamma(e-a)\Gamma(e-c)}.
\end{multline}
\end{theorem}
\begin{proof}
Writing $b-d+1=\epsilon$ (with the intention to let $\epsilon\to0$) in formula \eqref{eq:killbottom-1} brings it to the form:
\begin{equation}\label{eq:epsilonform}
{}_{4}F_3\!\!\left(\begin{matrix}a,b,c,f+1\\b+1-\epsilon,e,f\end{matrix}\right)
=M(\epsilon)
{}_{4}F_3\!\!\left(\begin{matrix}a,b,c,\eta+1\\b-\epsilon,e,\eta\end{matrix}\right),
\end{equation}
where
\begin{multline*}
M(\epsilon)=\frac{\epsilon(c+\epsilon-b)+a(c+\epsilon)-abc/f}{\epsilon(c+\epsilon-b)+a(c+\epsilon)-abc/(b-\epsilon)}
\\
=\frac{(b-\epsilon)(\epsilon^2+\epsilon(c-b+a)+ac(1-b/f))}{(b-\epsilon)(\epsilon^2+\epsilon(c-b+a)+ac)-abc}
\\
=\frac{1}{\epsilon}\frac{abc(1-b/f)+\epsilon(b(c-b+a)-ac(1-b/f))+\epsilon^2(2b-a-c)-\epsilon^3}
{b(c-b+a)-ac+\epsilon(2b-a-c)-\epsilon^2}
\\
=\frac{1}{\epsilon}\frac{abc(b-f)+O(\epsilon)}
{f(b-a)(b-c)+O(\epsilon)}=\frac{1}{\epsilon}\frac{abc(b-f)}
{f(b-a)(b-c)}(1+O(\epsilon))
\end{multline*}
and
$$
\eta=\frac{f\epsilon(c-b+\epsilon)+af(c+\epsilon)-abc}{(e-a-c-1-\epsilon)(b-\epsilon-f)}.
$$
Further, by \cite[(13)]{KPG} and the above expression for $\eta$, we obtain:
\begin{multline*}
{}_{4}F_3\!\!\left(\begin{matrix}a,b,c,\eta+1\\b-\epsilon,e,\eta\end{matrix}\right)
={}_{3}F_2\!\!\left(\begin{matrix}a,b,c\\b-\epsilon,e\end{matrix}\right)
+\frac{abc}{(b-\epsilon)e\eta}{}_{3}F_2\!\!\left(\begin{matrix}a+1,b+1,c+1\\b-\epsilon+1,e+1\end{matrix}\right)
\\
={}_{3}F_2\!\!\left(\begin{matrix}a,b,c\\b-\epsilon,e\end{matrix}\right)
+\frac{abc(e-a-c-1-\epsilon)(b-\epsilon-f)}{(b-\epsilon)e(f\epsilon(c-b+\epsilon)+af(c+\epsilon)-abc)}
{}_{3}F_2\!\!\left(\begin{matrix}a+1,b+1,c+1\\b-\epsilon+1,e+1\end{matrix}\right)
\\
={}_{3}F_2\!\!\left(\begin{matrix}a,b,c\\b-\epsilon,e\end{matrix}\right)+\frac{1+a+c-e}{e} {}_{3}F_2\!\!\left(\begin{matrix}a+1,b+1,c+1\\b-\epsilon+1,e+1\end{matrix}\right)
\\
+\left\{\frac{f(b-c)(a^2+(a-b)(c-e+1))+ab^2(c-f)}
{abce(b-f)}\epsilon +O(\epsilon^2)\right\}
 {}_{3}F_2\!\!\left(\begin{matrix}a+1,b+1,c+1\\b-\epsilon+1,e+1\end{matrix}\right).
\end{multline*}
Next,  note that as $\epsilon\to0$ we have
$$
\frac{(b)_{k}}{(b-\epsilon)_{k}}=1+\epsilon\sum_{j=0}^{k-1}\frac{1}{b+j}+O(\epsilon^2)
=1+\epsilon(\psi(b+k)-\psi(b))+O(\epsilon^2),
$$
so that
\begin{multline*}
{}_{3}F_2\!\!\left(\begin{matrix}a,b,c\\b-\epsilon,e\end{matrix}\right)={}_{2}F_1\!\!\left(\begin{matrix}a,c\\e\end{matrix}\right)
+\epsilon\sum_{k=0}^{\infty}\frac{(a)_k(c)_{k}}{(e)_kk!}(\psi(b+k)-\psi(b))+O(\epsilon^2)
\\
=\frac{\Gamma(e)\Gamma(e-a-c)}{\Gamma(e-a)\Gamma(e-c)}+\epsilon\cdot {}_{2}\hat{F}_1\!\!\left(\left.\begin{matrix}a,c\\e\end{matrix}\right\vert b\right)+O(\epsilon^2),
\end{multline*}
where we used the Gauss formula for ${}_{2}F_1(1)$ and the definition given in \eqref{eq:2F1hat}.

Substituting $M(\epsilon)$ and the above calculations into \eqref{eq:epsilonform} we will finally have
\begin{multline*}
 {}_{4}F_3\!\!\left(\begin{matrix}a,b,c,f+1\\b+1-\epsilon,e,f\end{matrix}\right)
=\frac{1}{\epsilon}\frac{abc(b-f)}{f(b-a)(b-c)}(1+O(\epsilon))
\\
\left\{
\epsilon\cdot {}_{2}\hat{F}_1\!\!\left(\left.\begin{matrix}a,c\\e\end{matrix}\right\vert b\right)+O(\epsilon^2)
+\frac{1+a+c-e}{e}\epsilon\cdot {}_{2}\hat{F}_1\!\!\left(\left.\begin{matrix}a+1,c+1\\e+1\end{matrix}\right\vert b+1\right)+O(\epsilon^2)\right\}
\\
+\frac{1}{\epsilon}\frac{abc(b-f)}{f(b-a)(b-c)}(1+O(\epsilon))\left\{\frac{f(b-c)(a^2+(a-b)(c-e+1))+ab^2(c-f)}
{abce(b-f)}\epsilon +O(\epsilon^2)\right\}
\\
\times
\left\{{}_{2}F_1\!\!\left(\begin{matrix}a+1,c+1\\e+1\end{matrix}\right)+O(\epsilon)\right\}
\\
=
\frac{abc(b-f)}{f(b-a)(b-c)}{}_{2}\hat{F}_1\!\!\left(\left.\begin{matrix}a,c\\e\end{matrix}\right\vert b\right)
+\frac{abc(b-f)(1+a+c-e)}{ef(b-a)(b-c)}{}_{2}\hat{F}_1\!\!\left(\left.\begin{matrix}a+1,c+1\\e+1\end{matrix}\right\vert b+1\right)
\\
+\frac{f(b-c)(a^2+(a-b)(c-e+1))+ab^2(c-f)}{f(b-a)(b-c)}\frac{\Gamma(e)\Gamma(e-a-c-1)}{\Gamma(e-a)\Gamma(e-c)}+O(\epsilon).
\end{multline*}
Letting $\epsilon\to0$ we arrive at \eqref{eq:4F3to2F1hat}.
\end{proof}

Setting  $f=a$ in \eqref{eq:4F3to2F1hat} and replacing $a+1\to{a}$ we obtain
\begin{corollary}\label{cr:3F2-psi}
The following identity holds:
 \begin{multline}\label{eq:3F2-psi}
{}_{3}F_2\!\!\left(\begin{matrix}a,c,b\\e,b+1\end{matrix}\right)
\!=\!\frac{bc}{b-c}
\left\{{}_{2}\hat{F}_1\!\!\left(\left.\begin{matrix}a-1,c\\e\end{matrix}\right\vert b\right)
+\frac{a+c-e}{e}{}_{2}\hat{F}_1\!\!\left(\left.\begin{matrix}a,c+1\\e+1\end{matrix}\right\vert b+1\right)\!\right\}
\\
+\left(1+\frac{c^2}{(b-c)(e-a)}\right)\frac{\Gamma(e)\Gamma(e-a-c)}{\Gamma(e-a)\Gamma(e-c)}.
\end{multline}
\end{corollary}

\paragraph{Acknowledgements.} We thank Asena \c{C}etinkaya for assistance in numerical verification of the results.

This work has been supported by the Ministry of Science and
Higher Education of the Russian Federation (agreement No. 075-02-2022-880). The second named author was also supported by RFBR (project 20-01-00018).

\end{document}